\documentclass[a4paper]{amsart}

\usepackage{amsmath}
\usepackage{amsfonts}
\usepackage{amssymb}
\usepackage{amsthm}
\usepackage{tikz}
\usetikzlibrary{matrix,arrows}

\theoremstyle{plain}
  \newtheorem{thm}{Theorem}[section]

\theoremstyle{definition}

\def\fr{\mathfrak}
\def\li{\langle}
\def\ri{\rangle}

\begin{document}

\title[Moduli spaces and Morita equivalence]{Moduli spaces of flat connections and Morita equivalence of quantum tori}
\author{Pavol \v Severa}
\address{Section of Mathematics, University of Geneva, 2-4 rue du Li\`evre,
c.p. 64, 1211 Gen\`eve 4, Switzerland, on leave from FMFI UK Bratislava, Slovakia}
\email{pavol.severa@gmail.com}
\maketitle

\begin{abstract}
We study moduli spaces of flat connections on surfaces with boundary, with boundary conditions given by Lagrangian Lie subalgebras. The resulting symplectic manifolds are closely related with Poisson-Lie groups and their algebraic structure (such as symplectic groupoid structure) gets a geometrical explanation via 3-dimensional cobordisms. We give a formula for the symplectic form in terms of holonomies, based on a central extension of the gauge group by closed 2-forms. This construction is finally used for a certain extension of the Morita equivalence of quantum tori to the world of Poisson-Lie groups.
\end{abstract}


\newcommand{\picSqOne}{
\begin{tikzpicture}[very thick]
  \fill[lime!30!white] (0,0) rectangle (1,1);
  \draw[red] (0,0)--node[below]{$\mathfrak{h}_1$}(1,0)
             (0,1)--node[above]{$\mathfrak{h}_1$}(1,1);
  \draw[blue] (0,0)--node[left]{$\mathfrak{h}_2$}(0,1)
              (1,0)--node[right]{$\mathfrak{h}_2$}(1,1);
  \foreach \x in {(0,0),(1,0),(0,1),(1,1)}
    \fill \x circle (0.03cm);
\end{tikzpicture}
}


\newcommand{\picMul}{
\begin{tikzpicture}[thick,join=round]
  \path[fill=blue,draw=blue] (0,2)..controls +(0.5,-0.75)..++(1,0)
       --++(0.5,0.5)..controls +(-0.5,-0.75)..++(-1,0);
  \draw[very thick] (1.5,2.5)..controls +(-0.5,-0.75)..++(-1,0);
  \fill[lime!30!white] (-1,2)--++(1,0)--++(0.5,0.5)--++(-1,0);
  \fill[lime!30!white] (1,2)--++(1,0)--++(0.5,0.5)--++(-1,0);
  \path[fill=red,draw=red] (0,0)..controls +(-0.25,1)..(-1,2)--(0,2)..controls +(0.5,-0.75)..(1,2)
              --(2,2)..controls +(-0.75,-1)..(1,0)--cycle;
  \path[fill=blue,draw=blue] (1.5,0.5)..controls +(0.25,1)..++(1,2)--++(-0.5,-0.5)
              ..controls +(-0.75,-1)..(1,0)--cycle;

  \draw[red] (-1,2) ++(0.5,0.5)--+(1,0)
              (1,2) ++(0.5,0.5)--+(1,0);
  \draw[blue] (-1,2)--++(0.5,0.5) (0,2)--++(0.5,0.5)
              (1,2)--++(0.5,0.5) (2,2)--++(0.5,0.5);

  \draw[very thick] (0,0)..controls +(-0.25,1)..(-1,2) (0,2)..controls +(0.5,-0.75)..(1,2)
              (1,0) ++(0.5,0.5)..controls +(0.25,1)..++(1,2) ++(-0.5,-0.5)
              ..controls +(-0.75,-1)..(1,0);
  \foreach \x in {(0,0),(1,0),(1.5,0.5),
                  (-1,2),(0,2),(0.5,2.5),(-0.5,2.5),
                  (1,2),(2,2),(2.5,2.5),(1.5,2.5)}
        \fill \x circle (0.025cm);

\end{tikzpicture}
}


\newcommand{\picEx}{
\begin{tikzpicture}[very thick]
 \fill[even odd rule,lime!30!white]
        (1,0) {coordinate (a)} arc (0:60:2cm) {coordinate (b)}
          arc (120:180:2cm) {coordinate(c)} arc (240:300:2cm);
  \fill[white] (0,0.2) {coordinate (d)} ..controls+(0.6,-0.4)and+(0.5,0.3)..(0,1.1) {coordinate (e)} (d)..controls+(-0.6,-0.4)and+(-0.5,0.3)..(e)--cycle;

  \draw[red] (a) arc (0:60:2cm);
  \draw[blue] (b) arc (120:180:2cm);
  \draw[violet] (c) arc (240:300:2cm);
  \draw[red] (d)..controls+(0.6,-0.4)and+(0.5,0.3)..(e);
  \draw[blue] (d)..controls+(-0.6,-0.4)and+(-0.5,0.3)..(e);

  \foreach \x in {(a),(b),(c),(d),(e)}
    \fill \x circle (0.03cm);
\end{tikzpicture}
}


\newcommand{\picExTot}{
\begin{tikzpicture}[very thick]
\begin{scope}
\fill[even odd rule,lime!30!white]
        (1,0) {coordinate (a)} arc (0:60:2cm) {coordinate (b)}
          arc (120:180:2cm) {coordinate(c)} arc (240:300:2cm);
  \fill[white] (0,0.2) {coordinate (d)} ..controls+(0.6,-0.4)and+(0.5,0.3)..(0,1.1) {coordinate (e)} (d)..controls+(-0.6,-0.4)and+(-0.5,0.3)..(e);

  \draw[red] (a) arc (0:60:2cm);
  \draw[blue] (b) arc (120:180:2cm);
  \draw[violet] (c) arc (240:300:2cm);
  \draw[red] (d)..controls+(0.6,-0.4)and+(0.5,0.3)..(e);
  \draw[blue] (d)..controls+(-0.6,-0.4)and+(-0.5,0.3)..(e);

  \draw[->,very thick,gray] (1.75,0.75)--(2.25,0.75);

  \foreach \x in {(a),(b),(c),(d),(e)}
    \fill \x circle (0.03cm);
\end{scope}

\begin{scope}[xshift=4cm]
   \fill[even odd rule,lime!30!white]
        (1,0) {coordinate (a)} arc (0:60:2cm) {coordinate (b)}
          arc (120:180:2cm) {coordinate(c)} arc (240:300:2cm);
  \fill[white] (0,0.2) {coordinate (d)} ..controls+(0.6,-0.4)and+(0.5,0.3)..(0,1.1) {coordinate (e)} (d)..controls+(-0.6,-0.4)and+(-0.5,0.3)..(e);

  \draw[red] (a) arc (0:60:2cm);
  \draw[blue] (b) arc (120:180:2cm);
  \draw[violet] (c) arc (240:300:2cm);
  \draw[red] (d)..controls+(0.6,-0.4)and+(0.5,0.3)..(e);
  \draw[blue] (d)..controls+(-0.6,-0.4)and+(-0.5,0.3)..(e);
  \draw[dotted] (b)--(e);

  \draw[->,very thick,gray] (1.75,0.75)--(2.25,0.75);

  \foreach \x in {(a),(b),(c),(d),(e)}
    \fill \x circle (0.03cm);

\end{scope}

\begin{scope}[xshift=8cm,yshift=0.75cm]
    \fill[lime!30!white] (90:1cm)--(141:1cm)--(193:1cm)--(243:1cm)--(297:1cm)--(347:1cm)--(399:1cm)--cycle;
  \draw[->,red] (399:1cm)--node[auto,swap]{$r_1$}(90:1cm);
  \draw[->,red]  (297:1cm)--node[auto,swap]{$r_2$}(347:1cm);
  \draw[->,blue] (90:1cm)--node[auto,swap]{$b_1$}(141:1cm);
  \draw[->,blue](193:1cm)--node[auto,swap]{$b_2$}(243:1cm);
  \draw[->,violet] (243:1cm)--node[auto,swap]{$v$}(297:1cm);
  \draw[->,dotted] (141:1cm)node[minimum size=0.4cm](a){}-- node[auto,left]{\it g\/}(193:1cm);
  \draw[->,dotted](347:1cm)-- node[auto,right]{$g^{-1}$}(399:1cm)node(b)[minimum size=0.4cm]{};
  \draw[white!85!black,ultra thick,<->] (a.south west).. controls +(-1.5,2) and +(1.5,2)..(b.south east);
  \node (f) at (0,-1.5) {$b_1\, g\, b_2\, v\, r_2\, g^{-1} r_1 =1$};

\end{scope}

\end{tikzpicture}
}


\newcommand{\picSqTwo}{
\begin{tikzpicture}[very thick]
  \fill[lime!30!white] (0,0) rectangle (1,1);
  \draw[red,->] (0,0)--node[below]{$r_1$}(1,0);
  \draw[red,->] (0,1)--node[above]{$r_2$}(1,1);
  \draw[blue,->] (0,0)--node[left]{$b_2$}(0,1);
  \draw[blue,->] (1,0)--node[right]{$b_1$}(1,1);
\end{tikzpicture}
}


\newcommand{\picSqThree}{
\begin{tikzpicture}[very thick]
  \fill[lime!30!white] (0,0) rectangle (1,1);
  \draw[red] (0,0)--node[below]{$\mathfrak{r}$}(1,0);
  \draw[violet] (0,1)--node[above]{$\mathfrak{v}$}(1,1);
  \draw[blue] (0,0)--node[left]{$\mathfrak{b}$}(0,1)
              (1,0)--node[right]{$\mathfrak{b}$}(1,1);
  \foreach \x in {(0,0),(1,0),(0,1),(1,1)}
    \fill \x circle (0.03cm);

  \begin{scope}[xshift=3cm]
  \fill[lime!30!white] (0,0) rectangle (1,1);
  \draw[red,->] (0,0)--node[below]{$r$}(1,0);
  \draw[violet,->] (0,1)--node[above]{$v$}(1,1);
  \draw[blue,->] (0,0)--node[left]{$b_2$}(0,1);
  \draw[blue,->] (1,0)--node[right]{$b_1$}(1,1);
  \end{scope}
\end{tikzpicture}
}


\newcommand{\picTri}{
\begin{tikzpicture}[very thick]
  \fill[lime!30!white] (-30:1)--(90:1)--(210:1);
  \draw[red] (210:1)--node[below]{$\mathfrak{r}$}(-30:1);
  \draw[violet] (-30:1)--node[auto,swap]{$\mathfrak{v}$}(90:1);
  \draw[blue] (90:1)--node[auto,swap]{$\mathfrak{b}$}(210:1);
  \foreach \x in {(-30:1),(90:1),(210:1)}
    \fill \x circle (0.03cm);

  \begin{scope}[xshift=4cm,->]
  \fill[lime!30!white] (-30:1)--(90:1)--(210:1);
  \draw[red] (210:1)--node[below]{$r$}(-30:1);
  \draw[violet] (-30:1)--node[auto,swap]{$v$}(90:1);
  \draw[blue] (90:1)--node[auto,swap]{$b$}(210:1);

  \end{scope}
\end{tikzpicture}
}


\newcommand{\picDouble}{
\begin{tikzpicture}[very thick]
 \fill[even odd rule,lime!30!white]
    (0,0)circle(0.7cm) (0,0)circle(1.3cm);

  \draw[red] (0,-0.7) arc (-90:90:0.7cm) (0.45,0)node{$\fr{r}$};
  \draw[blue] (0,0.7) arc (90:270:0.7cm) (-0.45,0)node{$\fr{b}$};
  \draw[red] (0,-1.3) arc (-90:90:1.3cm) (1.5,0)node{$\fr{r}$};
  \draw[blue] (0,1.3) arc (90:270:1.3cm) (-1.5,0)node{$\fr{b}$};
  \foreach \x in {(0,-0.7),(0,0.7),(0,1.3),(0,-1.3)}
    \fill \x circle (0.03cm);

  \begin{scope}[xshift=5cm,->]
  \fill[even odd rule,lime!30!white]
    (0,0)circle(0.7cm) (0,0)circle(1.3cm);
  \draw[red] (0,0.7) arc (90:-90:0.7cm);
  \node[red] at (0.45,0) {$r_1$};
  \draw[blue] (0,-0.7) arc (270:90:0.7cm);
  \node[blue] at (-0.45,0) {$b_1$};
  \draw[red] (0,1.3) arc (90:-90:1.3cm);
  \node[red] at (1.5,0) {$r_2$};
  \draw[blue] (0,-1.3) arc (270:90:1.3cm);
  \node[blue] at (-1.5,0) {$b_2$};
  \draw[dotted] (0,0.7)--node[right]{$g$}(0,1.3);
  \end{scope}

\end{tikzpicture}
}


\newcommand{\picSqGlue}{
\begin{tikzpicture}[very thick]
  \fill[lime!30!white] (0,0) rectangle (1,1);
  \draw[red] (0,0)--node[below]{$\mathfrak{r}$}(1,0)
             (0,1)--node[above]{$\mathfrak{r}$}(1,1);
  \draw[blue] (0,0)--node[left]{$\mathfrak{b}$}(0,1)
              (1,0)--node[right]{$\mathfrak{b}$}(1,1);
  \foreach \x in {(0,0),(1,0),(0,1),(1,1)}
    \fill \x circle (0.03cm);

 \begin{scope}[xshift=1.45cm]
  \fill[lime!30!white] (0,0) rectangle (1,1);
  \draw[red] (0,0)--node[below]{$\mathfrak{r}$}(1,0)
             (0,1)--node[above]{$\mathfrak{r}$}(1,1);
  \draw[blue] (0,0)--(0,1)
              (1,0)--node[right]{$\mathfrak{b}$}(1,1);
  \foreach \x in {(0,0),(1,0),(0,1),(1,1)}
    \fill \x circle (0.03cm);
 \end{scope}

\end{tikzpicture}
}


\newcommand{\picSqTwoOne}{
\begin{tikzpicture}[very thick]
  \fill[lime!30!white] (0,0) rectangle (1,1);
  \draw[red] (0,0)--node[below]{$\mathfrak{r}$}(1,0);
  \draw[blue] (1,0)--node[right]{$\mathfrak{b}$}(1,1)
             (0,0)--node[left]{$\mathfrak{b}$}(0,1);
  \draw[violet] (0,1)--node[above]{$\mathfrak{v}$}(1,1);
  \foreach \x in {(0,0),(1,0),(0,1),(1,1)}
    \fill \x circle (0.03cm);
\end{tikzpicture}
}


\newcommand{\picSqGlueTwo}{
\begin{tikzpicture}[very thick]
  \fill[lime!30!white] (0,0) rectangle (1,1);
  \draw[red] (0,0)--node[below]{$\mathfrak{r}$}(1,0);
  \draw[violet]    (0,1)--node[above]{$\mathfrak{v}$}(1,1);
  \draw[blue] (0,0)--node[left]{$\mathfrak{b}$}(0,1)
              (1,0)--node[right]{$\mathfrak{b}$}(1,1);
  \foreach \x in {(0,0),(1,0),(0,1),(1,1)}
    \fill \x circle (0.03cm);

 \begin{scope}[xshift=1.45cm]
  \fill[lime!30!white] (0,0) rectangle (1,1);
  \draw[red] (0,0)--node[below]{$\mathfrak{r}$}(1,0);
  \draw[violet]    (0,1)--node[above]{$\mathfrak{v}$}(1,1);
  \draw[blue] (0,0)--(0,1)
              (1,0)--node[right]{$\mathfrak{b}$}(1,1);
  \foreach \x in {(0,0),(1,0),(0,1),(1,1)}
    \fill \x circle (0.03cm);
 \end{scope}

\end{tikzpicture}
}


\newcommand{\picSqOneOne}{
\begin{tikzpicture}[very thick]
  \fill[lime!30!white] (0,0) rectangle (1,1);
  \draw[red] (0,0)--node[below]{$\mathfrak{r}$}(1,0)
             (0,1)--node[above]{$\mathfrak{r}$}(1,1);
  \draw[blue] (0,0)--node[left]{$\mathfrak{b}$}(0,1)
              (1,0)--node[right]{$\mathfrak{b}$}(1,1);
  \foreach \x in {(0,0),(1,0),(0,1),(1,1)}
    \fill \x circle (0.03cm);
\end{tikzpicture}
}


\newcommand{\picSqGlueThree}{
\begin{tikzpicture}[very thick]
  \fill[lime!30!white] (0,0) rectangle (1,1);
  \draw[red] (0,0)--node[below]{$\mathfrak{r}$}(1,0);
  \draw[violet]    (0,1)--node[above]{$\mathfrak{v}$}(1,1);
  \draw[blue] (0,0)--node[left]{$\mathfrak{b}$}(0,1)
              (1,0)--node[right]{$\mathfrak{b}$}(1,1);
  \foreach \x in {(0,0),(1,0),(0,1),(1,1)}
    \fill \x circle (0.03cm);

 \begin{scope}[yshift=-1.45cm]
  \fill[lime!30!white] (0,0) rectangle (1,1);
  \draw[red] (0,0)--node[below]{$\mathfrak{r}$}(1,0)
             (0,1)--(1,1);
  \draw[blue] (0,0)--node[left]{$\mathfrak{b}$}(0,1)
              (1,0)--node[right]{$\mathfrak{b}$}(1,1);
  \foreach \x in {(0,0),(1,0),(0,1),(1,1)}
    \fill \x circle (0.03cm);
 \end{scope}

\end{tikzpicture}
}


\newcommand{\picSHMor}{
\begin{tikzpicture}[very thick]
  \fill[lime!30!white] (0,0) rectangle (1,1);
  \draw[red] (0,0)--node[below]{$\mathfrak{r}$}(1,0);
  \draw[violet] (0,1)--node[above]{$\mathfrak{v}$}(1,1);
  \draw[blue] (0,0)--node[left]{$\mathfrak{b}$}(0,1)
              (1,0)--node[right]{$\mathfrak{b}$}(1,1);
  \foreach \x in {(0,0),(1,0),(0,1),(1,1)}
    \fill \x circle (0.03cm);
  \node at (-1,0.5) {$A\rightsquigarrow{}$};

  \begin{scope}[xshift=3.5cm]
  \fill[lime!30!white] (0,0) rectangle (1,1);
  \draw[blue] (0,0)--node[below]{$\mathfrak{b}$}(1,0);
  \draw[violet] (0,1)--node[above]{$\mathfrak{v}$}(1,1);
  \draw[red] (0,0)--node[left]{$\mathfrak{r}$}(0,1)
              (1,0)--node[right]{$\mathfrak{r}$}(1,1);
  \foreach \x in {(0,0),(1,0),(0,1),(1,1)}
    \fill \x circle (0.03cm);
  \node at (-1,0.5) {$B\rightsquigarrow{}$};
  \end{scope}

  \begin{scope}[xshift=7cm]
  \fill[lime!30!white] (0,0) rectangle (1,1);
  \draw[blue] (0,0)--node[below]{$\mathfrak{b}$}(1,0);
  \draw[blue] (0,1)--node[above]{$\mathfrak{b}$}(1,1);
  \draw[red] (0,0)--node[left]{$\mathfrak{r}$}(0,1)
              (1,0)--node[right]{$\mathfrak{r}$}(1,1);
  \foreach \x in {(0,0),(1,0),(0,1),(1,1)}
    \fill \x circle (0.03cm);
  \node at (-1,0.5) {$H\rightsquigarrow{}$};
  \end{scope}

  \begin{scope}[xshift=4cm, yshift=-2cm]
  \fill[lime!30!white] (30:1)--(-90:1)--(150:1);
  \draw[red] (-90:1)--node[auto,swap]{$\mathfrak{r}$}(30:1);
  \draw[violet] (30:1)--node[auto,swap]{$\mathfrak{v}$}(150:1);
  \draw[blue] (-90:1)--node[auto]{$\mathfrak{b}$}(150:1);
  \foreach \x in {(30:1),(-90:1),(150:1)}
    \fill \x circle (0.03cm);
  \node at (-1.5,-0.2) {$M\rightsquigarrow{}$};
  \end{scope}
\end{tikzpicture}
}


\newcommand{\picGammaZZ}{
\begin{tikzpicture}[very thick]
  \fill[lime!30!white] (0,0) rectangle (1,2);
  \draw[blue] (0,0)--node[below]{$\mathfrak{b}$}(1,0);
  \draw[violet] (0,2)--node[above]{$\mathfrak{v}$}(1,2);
  \draw[red] (0,0)--node[left]{$\mathfrak{r}$}(0,1)
              (1,0)--node[right]{$\mathfrak{r}$}(1,1);
  \draw[blue] (0,1)--node[left]{$\mathfrak{b}$}(0,2)
              (1,1)--node[right]{$\mathfrak{b}$}(1,2);

  \foreach \x in {(0,0),(1,0),(0,2),(1,2),(0,1),(1,1)}
    \fill \x circle (0.03cm);
  \node at (-1,1) {$\Gamma_{00}\sim{}$};
\end{tikzpicture}
}


\newcommand{\picGammaoo}{
\begin{tikzpicture}[very thick]
  \fill[lime!30!white] (0,0) rectangle (1,1);
  \draw[red] (0,0)--node[below]{$\mathfrak{r}$}(1,0);
  \draw[violet] (0,1)--node[above]{$\mathfrak{v}$}(1,1);
  \draw[blue] (0,0)--node[left]{$\mathfrak{b}$}(0,1)
              (1,0)--node[right]{$\mathfrak{b}$}(1,1);
  \foreach \x in {(0,0),(1,0),(0,1),(1,1)}
    \fill \x circle (0.03cm);
  \node at(-1,0.5) {$\Gamma_{11}\sim{}$};
\end{tikzpicture}
}


\newcommand{\picGammaZo}{
\begin{tikzpicture}[very thick]
  \fill[lime!30!white] (0,0)--(1,0.5)--(1,1.5)--(0,2)--cycle;
  \draw[blue] (0,0)--node[auto,swap]{$\mathfrak{b}$}(1,0.5);
  \draw[violet] (0,2)--node[auto]{$\mathfrak{v}$}(1,1.5);
  \draw[red] (0,0)--node[left]{$\mathfrak{r}$}(0,1)
              (1,0.5)--node[right]{$\mathfrak{r}$}(1,1.5);
  \draw[blue] (0,1)--node[left]{$\mathfrak{b}$}(0,2);

  \foreach \x in {(0,0),(1,0.5),(1,1.5),(0,2),(0,1)}
    \fill \x circle (0.03cm);
  \node at (-1,1) {$\Gamma_{01}\sim{}$};
\end{tikzpicture}
}


\newcommand{\picTriTwo}{
\begin{tikzpicture}[very thick]
  \fill[lime!30!white] (30:1)--(-90:1)--(150:1);
  \draw[red] (-90:1)--node[auto,swap]{$\mathfrak{r}$}(30:1);
  \draw[violet] (30:1)--node[auto,swap]{$\mathfrak{v}$}(150:1);
  \draw[blue] (-90:1)--node[auto]{$\mathfrak{b}$}(150:1);
  \foreach \x in {(30:1),(-90:1),(150:1)}
    \fill \x circle (0.03cm);
\end{tikzpicture}
}

\section{Introduction}

Let $\mathfrak{g}$ be a Lie algebra with an invariant inner product
$\langle\cdot,\cdot\rangle$ (of any signature). It gives rise to two interesting types of
symplectic manifolds. The first type are moduli spaces of flat $\mathfrak{g}$-connections on oriented surfaces. The second type are symplectic manifolds
connected with Poisson-Lie groups such as the Lu-Weinstein double symplectic groupoid \cite{LW} (the symplectic groupoid integrating a Poisson-Lie group) corresponding to a Manin triple
$$\mathfrak{h}_1,\mathfrak{h}_2\subset\mathfrak{g}.$$

We shall notice that these "Poisson-Lie type" symplectic manifolds are, in fact, themselves moduli spaces of flat connections, if we allow surfaces with boundary and impose boundary conditions on the flat connections. To get the Lu-Weinstein double groupoid, the surface is a square, with boundary conditions as on the picture:
$$\picSqOne$$
We shall provide formulas for  symplectic forms using holonomies of the flat connections, in the spirit of Alekseev-Malkin-Meinrenken \cite{AMM}.
The basic idea is that the symplectic form can be interpreted as the integral over the surface of the curvature of a certain connection. The integral is then readily computed in terms of parallel transport.
 Moreover we shall describe how 3dim bodies give rise to Lagrangian submanifolds; for example, this picture gives one of the two products in Lu-Weinstein double groupoid:

$$\picMul$$
This is  a symplectic version of Chern-Simons TQFT in the sense of D.~Freed \cite{F}, with appropriate boundary conditions.

The motivation for this work was to give a symplectic description of Morita equivalence of quantum tori, and moreover, to extend this Morita equivalence from Abelian T-duality to Poisson-Lie T-duality \cite{KS} (though just on the symplectic level, without performing geometrical quantization). This is done in the final section.

\subsection*{Aknowledgements} I am grateful to Anton Alekseev, David Li-Bland, \v Stefan Sak\'alo\v s and Andr\'as Szenes for useful discussion and suggestions.

\section{Colored surfaces and moduli spaces}\label{sec:colsurf}
Let $G$ be a connected Lie group and $\li,\ri$ an Ad-invariant inner product (of any signature) on its Lie algebra $\mathfrak{g}$.

We shall consider compact oriented surfaces $\Sigma$ with corners (i.e.~locally looking like $(\mathbb{R}_{\geq 0})^2$). We shall assume that none of the components of $\Sigma$ is closed and that on each component of $\partial\Sigma$ there is at least one corner.\footnote{We impose these assumptions only for simplicity reasons, as they imply that the moduli spaces defined below are non-singular.} The boundary of $\Sigma$ is thus split by the corners into a finite number of arcs.

 For each arc $a$ we choose a  Lie subalgebra $\mathfrak{h}_{(a)}\subset\mathfrak{g}$ which is Lagrangian w.r.t.\ $\li,\ri$ (i.e. $\mathfrak{h}_{(a)}^\perp=\mathfrak{h}_{(a)}^{\vphantom{\perp}}$). Let $H_{(a)}\subset G$ be the corresponding connected Lie subgroup.
We demand for every corner $x$ of $\Sigma$ that if $a$ and $b$ are the arcs meeting at $x$ then
 $\mathfrak{h}_{(a)}\cap\mathfrak{h}_{(b)}=0$.
We shall call such a $\Sigma$ (together with the choice of subalgebras) a {\em colored surface}.
$$\picEx$$

For every colored surface $\Sigma$ we define a symplectic manifold $\mathcal{M}_\Sigma$. Let us first describe $\mathcal{M}_\Sigma$ in a way which depends on a choice of certain cuts of $\Sigma$.

We keep cutting $\Sigma$ along paths connecting corners until we get a polygon. For every  side $s$ of the polygon we choose an element $g_s\in G$ such that: 
\begin{enumerate}
\item if $s$ is an arc of the boundary of $\Sigma$ then $g_s\in H_{(s)}$
\item if $s$ and $s'$ are the two sides which are the result of a cut then $g_{s'}=g_{s}^{-1}$
\item the product of all $g_s$'s along the boundary of the polygon (in their natural cyclic order) is equal to 1.
\end{enumerate}
An assignment $s\mapsto g_s$ satisfying these properties is, by definition, a point in $\mathcal{M}_\Sigma$.
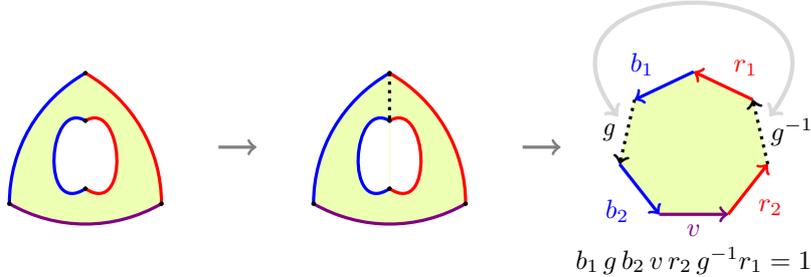
$$\picExTot$$

$\mathcal{M}_\Sigma$ can thus be described as the preimage of $1$ under a map
\begin{equation}
  \prod_a H_{(a)}\times G^{\mathrm{\#cuts}}\to G.\label{consmap}
\end{equation}
The map is a submersion and thus $\mathcal{M}_\Sigma$ is a manifold.

Let us now describe $\mathcal{M}_\Sigma$ without using cuts.
Let $X\subset\Sigma$ be the set of corners of $\Sigma$ and let $\Pi_1(\Sigma,X)$ be the fundamental groupoid of $\Sigma$ (the set of objects of $\Pi_1(\Sigma,X)$ is $X$ and morphisms are homotopy classes of paths between corners). Every arc of the boundary can be seen as a morphism in $\Pi_1(\Sigma,X)$. Then
$$\mathcal{M}_\Sigma=\{F:\Pi_1(\Sigma,X)\to G;\,F(a)\in H_{(a)}\text{ for every arc }a\}.$$

Finally, let us describe $\mathcal{M}_\Sigma$ as a moduli space of flat connections.
Let $\pi:P\to \Sigma$ be a principal $G$-bundle. For every arc $a$ we choose a reduction of $P|_a$ to $H_{(a)}\subset G$, i.e. a submanifold $P_{(a)}\subset\pi^{-1}(a)$ which is a principal $H_{(a)}$-bundle over $a$. For every corner $x\in\Sigma$ we choose a point  $p_x\in P_{(a)}\cap P_{(b)}$ where $a$ and $b$ are the arcs meeting at $x$.  Let us call $\pi:P\to \Sigma$ with its additional structure a \emph{colored $G$-bundle} over $\Sigma$.

 Let us consider connections on $P$ which restrict to connections (i.e. to $\mathfrak{h}_{(a)}$-valued 1-forms) on every $P_{(a)}$; we shall call such a connection a \emph{colored connection}. $\mathcal{M}_\Sigma$ can then be described the moduli space of colored flat connections on colored $G$-bundles over $\Sigma$. The groupoid morphism $\Pi_1(\Sigma,X)\to G$ corresponding to a colored flat connection is given by parallel transport (the fiber of $P$ over any corner $x$ is trivialized by the choice of the point $p_x$). 

If $P$ is the trivial $G$-bundle $P=\Sigma\times G$ and its coloring is also trivial (i.e.\ $P_{(a)}=a\times H_{(a)}$, $p_x=(x,1)$) then a colored connection can be described as a 1-form $A\in\Omega^1(\Sigma)\otimes\mathfrak{g}$ such that the restriction of $A$ to any arc $a$ is in $\Omega^1(a)\otimes\mathfrak{h}_{(a)}$. The space of these flat connections modulo the gauge transformations (by   maps $g:\Sigma\to G$ such that $g(x)=1$ for every corner $x$ and $g(a)\subset H_{(a)}$ for every arc $a$) is a connected component of $\mathcal{M}_\Sigma$.
\section{Symplectic form in terms of holonomies}

\subsection{Symplectic form on moduli spaces of flat connections}

Let $P\to\Sigma$ be a colored $G$-bundle. Colored connections on $P$ form an affine space $\mathcal{A}_\text{col}(P)$ modeled on $\Omega^1_\text{col}(\Sigma,\mathit{Ad}_P)$, where $\Omega_\text{col}(\Sigma,\mathit{Ad}_P)\subset\Omega(\Sigma,\mathit{Ad}_P)$ is the space of forms that restrict to $\Omega(a,\mathit{Ad}_{P_{a}})$ on every arc $a\subset\partial\Sigma$.

If $A$ is a flat colored connection on $P$ then the covariant differential $d_A$ makes $\Omega_\text{col}(\Sigma,\mathit{Ad}_P)$ to a complex and we have a natural isomorphism $$T_{[P,A]}\mathcal{M}_\Sigma\cong H^1(\Omega_\text{col}(\Sigma,\mathit{Ad}_P),d_A),$$
where $[P,A]\in\mathcal{M}_\Sigma$ denotes the isomorphism class of $(P,A)$.
The antisymmetric pairing 
$$\omega([\alpha],[\beta])=\int_\Sigma\langle\alpha\wedge\beta\rangle$$
on $T_{[P,A]}\mathcal{M}_\Sigma$ ($\alpha,\beta\in\Omega^1_\text{col}(\Sigma,\mathit{Ad}_P)$) is non-degenerate by Poincar\'e--Verdier duality. 

The moduli space $\mathcal{M}_\Sigma$ becomes in this way a symplectic manifold. If $U\subset\mathcal{M}_\Sigma$ is an open subset and $\phi:U\to\mathcal{A}_\text{col}(P)$, $\phi:x\mapsto A_x$, is a smooth family of colored flat connections such that $[P,A_x]=x$, then
\begin{equation}\label{sympl}
\omega=\phi^*\omega_{\mathcal{A}}
\end{equation}
where $\omega_{\mathcal{A}}\in\Omega^2(\mathcal{A}_\text{col}(P))$
is the constant (hence closed) 2-form
$$\omega_{\mathcal{A}}(\alpha,\beta)=\int_\Sigma\langle\alpha\wedge\beta\rangle$$
on the affine space $\mathcal{A}_\text{col}(P)$; $\omega$ is therefore closed.\footnote{The symplectic manifold $(\mathcal{M}_\Sigma,\omega)$ is best described as the symplectic reduction of $(\mathcal{A}_\text{col}(P),\omega_{\mathcal{A}})$ by the group of the automorphisms of $P$ preserving the coloring. Making this statement precise is, however, rather technical.} The symplectic form $\omega$ is a straightforward generalization of the symplectic form of Atiyah-Bott \cite{AB} and Goldman \cite{G} who considered closed surfaces.

\subsection{Central extension by closed 2-forms}

Let $M$ be a manifold and let $\Omega_\textit{closed}^2(M)$ denote the space of closed 2-forms on $M$. Let us recall that the Lie algebra $\mathfrak{g}(M)$ of smooth maps $M\to\mathfrak{g}$ has a central extension
$\tilde{\mathfrak{g}}(M)$
by $\Omega_\textit{closed}^2(M)$: as a vector space,
$$\tilde{\mathfrak{g}}(M)=\mathfrak{g}(M)\oplus\Omega_\textit{closed}^2(M),$$
and the bracket is
\begin{equation}\label{bracket}
  [(t_1,\omega_1),(t_2,\omega_2)]=([t_1,t_2],\li dt_1\wedge dt_2 \ri).
\end{equation}
 The corresponding group $\tilde G(M)$, a central extension of $G(M)$ (the group of smooth maps $M\to G$) by $\Omega_\textit{closed}^2(M)$, can be described as follows: its elements are pairs
\begin{equation}
  (g,\omega),\quad g:M\to G,\ \omega\in\Omega^2(M),\ d\omega=\frac{1}{2}g^*\eta\label{ex}
\end{equation}
where the invariant 3-form $\eta$ on $G$ is given by $$\eta(u,v,w)=\li[u,v],w\ri.$$
The product in the group is
\begin{equation}
  (g_1,\omega_1)(g_2,\omega_2)=(g_1g_2,\,\omega_1+\omega_2+\frac{1}{2}\li g_1^{-1}dg_1^{\vphantom{-1}}\wedge dg_2^{\vphantom{-1}}\,g_2^{-1}\ri)\label{exx}
\end{equation}
and the  inverse
\begin{equation}
  (g,\omega)^{-1}=(g^{-1},-\omega).\label{exxx}
\end{equation}

Finally, let us also introduce an auxiliary group  $\tilde{G}_\textit{big}(M)\supset\tilde{G}(M)$:
\begin{equation}\label{gbig}
  \tilde{G}_\textit{big}(M)=G(M)\times\Omega^2(M),
\end{equation}
with the product and inverse given by
 the same formulas (\ref{exx}), (\ref{exxx}). The map
$$\tilde{G}_\textit{big}(M)\to\Omega^3_\textit{closed}(M),\quad (g,\omega)\mapsto d\omega-\frac{1}{2}g^*\eta$$
is a group morphism and $\tilde{G}(M)$ is its kernel.

\subsection{Symplectic form in terms of holonomies}

Let us cut $\Sigma$ until we get a polygon (as in Section \ref{sec:colsurf}). For each side $s$ of the polygon we have a map $\gamma_s:\mathcal{M}_\Sigma\to G$  (the holonomy along the side).

\begin{thm}\label{thm:holon}
  The symplectic form $\omega$ on $\mathcal{M}_\Sigma$ is given by
  $$(1,\omega)=\prod_s (\gamma_s,0)$$
  where the product is taken in the group $\tilde{G}_\textit{big}(\mathcal{M}_\Sigma)$ and the sides of the polygon are taken in their natural (cyclic) order.
\end{thm}

The idea of the proof is that $\omega$ is the integral of the curvature of a $\tilde{\mathfrak{g}}(\mathcal{M}_\Sigma)$-valued connection on $\Sigma$, and hence can be expressed in terms of the holonomies $g_s$'s. The proof is in Section \ref{sec:holproof}.
The formula for $\omega$ is a generalization of a similar formula of Alekseev-Malkin-Meinrenken \cite{AMM} for the case of closed surfaces.

\subsection{Integral of curvature}

Let
$$C\to\tilde K\to K$$
be a central extension of Lie groups. Let $\tilde P\to D$ be a principal $\tilde K$-bundle over a disk $D$ and let $P\to D$ be the corresponding $K$-bundle, $P=\tilde P/C$. Suppose that $A$ is a flat connection on $P$ and $\tilde A$ is a (non-flat) connection on $\tilde P$ lifting $A$. 
 The curvature $\tilde F$ of $\tilde A$ is a $\mathfrak{c}$-valued 2-form on $D$ and its integral is
\begin{equation}\label{intcur}
  C \ni\exp\int_D \tilde F=\mathrm{hol}_{\partial D}\,\tilde A.
\end{equation}

The proof of this simple claim is obvious: trivialize $\tilde P\to D$ (and hence $P\to D$) in such a way that the connection $A$ on $P=D\times K$ becomes trivial.
 Such a trivialization can be achieved e.g.\ by  by the parallel transport of $\tilde A$ along straight lines starting at the center of the disc $D$. Formula \eqref{intcur} then becomes Stokes theorem.

We shall use (\ref{intcur}) for the central extension
$$\Omega_\textit{closed}^2(M)\to\tilde{G}(M)\to G(M).$$

\subsection{Symplectic form as integral of curvature}

Let $\Sigma$ be a colored surface. Let $U\subset\mathcal{M}_\Sigma$ be an open subset, $P\to\Sigma$ a colored $G$-bundle and $A_x$ a smooth family of colored flat connections on $P$ parametrized by $x\in U$, such that the class of $(P,A_x)$ is $x$. $\mathcal{M}_\Sigma$ can be covered by such open subsets $U$.

Using the inclusion $G\to \tilde G(U)$, $g\mapsto (g,0)$, we lift $P$ to a principal $\tilde G(U)$-bundle $\tilde P_U\to D$. Similarly, the inclusion $G\to G(M)$ lifts $P$ to a principal $G(U)$-bundle $P_U\to\Sigma$ and $P_U=\tilde P_U/\Omega^2_\textit{closed}(U)$.

The family $A_x$ can be seen as a flat connection $A$ on $P_U$, and $\tilde A=(A,0)$  as a (non-flat) connection on $\tilde P_U$. The curvature $\tilde F$ of $\tilde A$ is a $\Omega_\textit{closed}^2(U)$-valued 2-form on $\Sigma$, and the integral of $\tilde F$ is (using (\ref{sympl}) and (\ref{bracket})) the symplectic form on $U\subset \mathcal{M}_\Sigma$:
\begin{equation}\label{om-intF}
  \omega=\int_\Sigma\tilde F.
\end{equation}

\subsection{Symplectic form in terms of holonomies (proof)}\label{sec:holproof}


\begin{proof}[Proof of Theorem \ref{thm:holon}]
 It follows immediately from (\ref{om-intF}) and (\ref{intcur}). We cut $\Sigma$ to a polygon. If $s\subset\partial\Sigma$ then $(\gamma_s,0)$ is the holonomy of $\tilde A$ along $s$ (since $\mathfrak{h}_{(s)}$ is isotropic and thus the cocycle in (\ref{bracket}) vanishes). If $s$ comes from a cut then the holonomy of $\tilde A$ along $s$ is $(\gamma_s,\beta)$ for some 2-form $\beta$. However, the holonomy along the other side coming from the same cut is its inverse; we can thus replace $\beta$ with 0 and the product of holonomies will not change.
\end{proof}

\subsection{Examples}

Let $\Sigma$ be a square colored by a Manin triple $\mathfrak{r},\mathfrak{b}\subset\mathfrak{g}$:
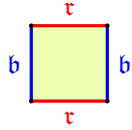
$$\picSqOneOne$$
Let us denote the holonomies as on the picture, i.e.
$$\mathcal{M}_\Sigma=\{(r_1,r_2,b_1,b_2)\in R\times R\times B\times B;\;r_1b_1=b_2r_2\}:
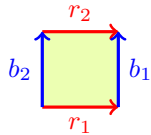
$$
$$\picSqTwo$$
In $\tilde G(\mathcal{M}_\Sigma)$ we have
$(r_1,0)(b_1,0)=(r_1b_1,\li  r_1^{-1}dr_1^{\vphantom{-1}}\wedge db_1^{\vphantom{-1}} b_1^{-1}\ri/2)$ and
$(b_2,0)(r_2,0)=(b_2r_2,\li  b_2^{-1}db_2^{\vphantom{-1}}\wedge dr_2^{\vphantom{-1}} r_2^{-1}\ri/2)$. The symplectic form on $\mathcal{M}_\Sigma$ is thus
$$\omega=\frac{1}{2}\li r_1^{-1}dr_1^{\vphantom{-1}}\wedge db_1^{\vphantom{-1}} b_1^{-1}\ri-\frac{1}{2}\li b_2^{-1}db_2^{\vphantom{-1}}\wedge dr_2^{\vphantom{-1}} r_2^{-1}\ri.$$
This symplectic manifold $(\mathcal{M}_\Sigma,\omega)$ is the Lu-Weinstein double symplectic groupoid \cite{LW} corresponding to the triple $R,B\subset G$. This fact was already noticed in \cite{Se2}. A similar interpretation of the Lu-Weinstein double groupoid was found by P.~Boalch in \cite{B} using irregular connections.

Let now $\Sigma$ be a square colored as follows:
$$\picSqThree$$
The symplectic form is
$$\omega=\frac{1}{2}\li r^{-1}dr\wedge db_1^{\vphantom{-1}} b_1^{-1}\ri-\frac{1}{2}\li b_2^{-1}db_2^{\vphantom{-1}}\wedge dv\,v^{-1}\ri.$$
The symplectic manifold $(\mathcal{M}_\Sigma,\omega)$ is again a well-known object: it is the symplectic groupoid integrating the homogeneous Poisson space given by $R,B,V\subset G$ via Drinfeld's classification \cite{D}. This symplectic groupoid was discovered by Jiang-Hua Lu \cite{L}.

Now let $\Sigma$ be a triangle.
$$\picTri$$
In this case
$$\mathcal{M}_\Sigma=\{(r,b,v)\in R\times B\times V;\;rbv=1\}.$$
The symplectic form is
$$\omega=\frac{1}{2}\li v^{-1}dv\wedge db\,b^{-1}\ri.$$
This symplectic manifold is, up to covering, the big symplectic leaf in the homogeneous Poisson space given by $R,B,V\subset G$. It will play a role when we discuss Morita equivalence.

Finally, let us discuss the simplest $\Sigma$ that requires a cut.
$$\picDouble
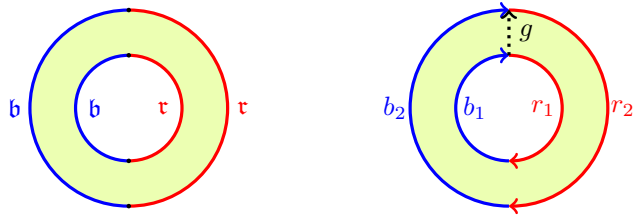
$$
$$\mathcal{M}_\Sigma=\{(r_1,r_2,b_1,b_2,g)\in R^2\times B^2\times G;\;
   r_1b_1g=g\,r_2b_2\}$$
$$(1,\omega)=(r_1,0)(b_1,0)(g,0)\left((g,0)(r_2,0)(b_2,0)\right)^{-1}$$
The symplectic manifold $(\mathcal{M}_\Sigma,\omega)$ is the double symplectic groupoid integrating the Drinfeld double given by the triple $R,B\subset G$.

\section{Painted bodies and Lagrangian submanifolds}

In this section we shall discuss how cobordisms of painted surfaces give rise to Lagrangian submanifolds in the moduli spaces. These Lagrangian submanifolds will turn the moduli spaces into interesting algebraic objects, such as (double) groupoids, modules, etc. The Lagrangian submanifold will consist of those flat connection on the surface that can be extended to flat connections on the 3dim manifold (cobordism). This construction is a straightforward generalization of the symplectic Chern-Simons theory of D.~Freed \cite{F}, who considered closed surfaces.

\subsection{Painted bodies}
Let us  consider a compact oriented 3dim manifold with corners (i.e.\ locally looking as $(\mathbb{R}_{\geq0})^3$). Its boundary is divided to vertices (corners), edges and faces. For some of the faces we choose a Lagrangian Lie subalgebra of $\fr{g}$ (we shall call such a face painted). We shall require the following. Whenever two faces meet along an edge then at least one of them is painted, and if both are painted, then the two subalgebras are transverse. At each vertex should meet two painted and one unpainted face. Finally, the unpainted part of the boundary should be a colored surface, i.e.\ each of its components should have boundary and on each of the boundary circles there should be a corner. We shall call such a manifold a \emph{painted body}. The unpainted part of the boundary of a body $X$ will be denoted $\Sigma_X$.

\subsection{Flat connections on a painted body}
Let $X$ be a painted body. We shall consider principal $G$-bundles $P\to X$ with a reduction to $H_{(a)}$ over every painted face $a$ and with a section over every edge between painted faces. We then consider flat connections compatible with the reductions.
%
%
%
Let $\mathcal{L}_X\subset\mathcal{M}_{\Sigma_X}$  denote the set of equivalence classes of flat colored connections on $\Sigma_X$ that are extensible to $X$. We shall call $\mathcal{L}_X$ \emph{smooth} if it is a submanifold and moreover it can be locally lifted to a smooth family of flat connections on $X$. 

If $\mathcal{L}_X$ is smooth then  it  is Lagrangian, as can be easily shown by a formal calculation of its tangent space.
Let $P$ be a painted $G$-bundle over $X$, $Ad_P\to X$ the vector bundle associated to the adjoint representation of $G$ on $\mathfrak{g}$, and let
$$\Omega_\text{col}(X,Ad_P)\subset \Omega(X,Ad_P)$$
be the space of $Ad_P$-valued differential forms that take values in the corresponding subalgebra of $\fr{g}$ when restricted to a painted face of $X$. Let $A$ be a flat colored connection; then $d_A$ makes
$\Omega_\text{col}(X,Ad_P)$ into a complex. Let us denote this complex
$\Omega_A(X)$
and its cohomology
$H_A(X)$.

Let $P'=P|_{\Sigma_X}$ and $A'=A|_{\Sigma_X}$.  Let us consider the short exact sequence
 $$0\to\Omega_{A,0}(X)\to\Omega_A(X)\to\Omega_{A'}(\Sigma_X)\to0$$
(where $\Omega_{A,0}(X)$ are the forms vanishing at $\Sigma_X$) and
the following piece of the resulting long exact sequence:
\begin{equation}\label{eq:les}
 H^1_A(X)\to H^1_{A'}(\Sigma_X)\to H^2_{A,0}(X).
\end{equation}
We have
$$H^1_{A'}(\Sigma_X)=T_{[P',A']}\mathcal{M}_{\Sigma_X},$$
and the image of the first arrow is the formal $T_{[P',A']}\mathcal{L}_X$.

By Poincar\'e duality the dual of \eqref{eq:les} is obtained just by reversing the arrows:
$$ H^2_{A,0}(X)\leftarrow H^1_{A'}(\Sigma_X)\leftarrow H^1_{A}(X)$$
(in particular, the identification of $H^1_{A'}(\Sigma_X)$ with its dual is via the symplectic form).
As a consequence, the image of the first arrow in \eqref{eq:les} is a Lagrangian subspace. 

If $\mathcal{L}_X$ is smooth then the formal tangent space is the true tangent space. We thus proved
\begin{thm}
 If $\mathcal{L}_X\subset\mathcal{M}_{\Sigma_X}$ is smooth, it is a Lagrangian submanifold.
\end{thm}

In all the examples that we consider below, $\mathcal{L}_X$ is easily seen to be smooth.

\subsection{Examples}



As we noticed above, Lu-Weinstein's double symplectic groupoid corresponding to a Manin triple $B,R\subset G$ is the moduli space for the surface
$$\picSqOneOne$$
The graph of one of the products in this double groupoid is $\mathcal{L}_X$ where $X$ is 
$$\picMul$$
In other words, the product is given by gluing squares along the adjacent sides on the picture
$$\picSqGlue$$
The other product is obtained when we exchange the colors.

The moduli space of
$$\picSqTwoOne$$
is a symplectic groupoid via the gluing
$$\picSqGlueTwo$$
It is a symplectic groupoid integrating the Poisson $B$-homogeneous space corresponding to the quadruple $R,B,V\subset G$.
It is also a module of
$$\picSqOneOne$$
via the gluing 
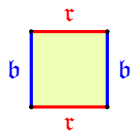
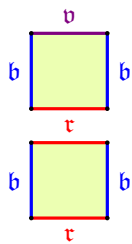
$$\picSqGlueThree$$

Similar pictures can be drawn for the double symplectic groupoid integrating the Drinfeld double and also for its $R$-matrix in the sense of Weinstein and Xu; see \cite{Se} for details.

\section{Morita equivalence of quantum tori and beyond}

This last section is a bit speculative. On the other hand, it describes the motivation for the constructions described above, so it is included anyway.

\subsection{Morita equivalence of quantum tori}

Recall that two algebras $A$ and $B$ are said to be Morita equivalent if their categories of modules are linearly equivalent. Equivalently, there exist a $A\otimes B^\text{op}$-module $M$ and a $B\otimes A^\text{op}$ module $N$ such that $M\otimes_B N\cong A$ and $N\otimes_A M\cong B$.

Let $\theta_{ij}=-\theta_{ji}$, $1\leq i,j\leq n$, be a skew-symmetric matrix with real elements. We suppose that the graph of the corresponding linear map $\mathbb{R}^n\to\mathbb{R}^n$ intersects $\mathbb{Z}^{2n}\subset\mathbb{R}^{2n}$ only in $0\in\mathbb{Z}^{2n}$. To the matrix $\theta$ we associate the algebra $\mathbb{T}^n_\theta$ (a quantum torus) generated by elements $u_i$ ($1\leq i\leq n$) and their inverses, modulo relations $u_i u_j =\exp(2\pi \sqrt{-1} \theta_{ij})u_j u_i$.
A famous result of Rieffel and Schwarz \cite{RS} says that the algebra $\mathbb{T}^n_\theta$ is Morita equivalent to $\mathbb{T}^n_{\theta^{-1}}$.\footnote{it implies a Morita equivalence of $\mathbb{T}^n_\theta$ with $\mathbb{T}^n_{A\cdot\theta}$ for any $A\in SO(n,n;\mathbb{Z})$, where $SO(n,n;\mathbb{Z})$ acts on the graph of $\theta$ in $\mathbb{R}^{2n}$, provided the transformed graph is again a graph}

The quantum torus $\mathbb{T}^n_\theta$ can be seen as a quantization of the $n$-dimensional torus $\mathbb{T}^n$ with the constant Poisson structure given by $\theta$. The following natural questions are due to A.\ Schwarz and A.\ Weinstein (motivated by an extension of $T$-duality \cite{Sch} to Poisson-Lie $T$-duality \cite{KS}).
\begin{enumerate}
\item Is there a generalization of Morita equivalence when $\mathbb{T}^n$ is replaced by a quantum group $H$ and $\mathbb{T}^n_\theta$ by a torsor of $H$?
\item Is there a symplectic/Poisson version of Morita equivalence for tori with constant Poisson structure? Can it be extended to Poisson-Lie groups, giving a symplectic/Poisson analog of Question 1?
\end{enumerate}

We shall give an answer to Question 2. It will provide a conjectural answer to Question 1.

\subsection{$H$-Morita equivalence}

Let $H$ be a Hopf algebra. Let $A$ be an associative algebra in the (monoidal) category $H\text{-mod}$ of left $H$-modules. In other words, $A$ is an $H$-module and the product $A\otimes A\to A$ is a morphism of $H$-modules.

Let $A\text{-mod}_H$ be the category of $A$-modules in $H\text{-mod}$, i.e.\ the category of vector spaces $V$ which are modules of both $A$ and $H$, such that $A\otimes V\to V$ is a morphism of $H$-modules. Let $A\text{-mod}$ be the category of $A$-modules in the category of vector spaces. We have the forgetful functor $\mathrm{res}:A\text{-mod}_H\to A\text{-mod}$ and its left adjoint $\mathrm{ind}:A\text{-mod}\to A\text{-mod}_H$.

Let now $B$ be an algebra in the (monoidal) category $H\text{-comod}$ of right $H$-comodules. We have the category $B\text{-mod}^H$ of $B$-modules in $H\text{-comod}$ and the category $B\text{-mod}$. Now we have the forgetful functor $\mathrm{cores}:B\text{-mod}^H\to B\text{-mod}$ and its right adjoint $\mathrm{coind}:B\text{-mod}\to B\text{-mod}^H$.

We shall say that $A$ and $B$ are \emph{$H$-Morita equivalent} if there are equivalences of linear categories $A\text{-mod}\to B\text{-mod}^H$ and $A\text{-mod}_H\to B\text{-mod}$ such that the diagram
\tikzset{diagram/.style={matrix of math nodes, row sep=3em, column sep=2.5em, text height=1.5ex, text depth=0.25ex}}
\[
 \begin{tikzpicture}
  \matrix(m)[diagram]{A\text{-mod} & B\text{-mod}^H \\
 A\text{-mod}_H & B\text{-mod} \\};
\draw[->]
 (m-1-1) edge (m-1-2)
         edge node[auto,swap]{$\mathrm{ind}$}(m-2-1)
 (m-2-1) edge  (m-2-2)
(m-1-2) edge node[auto]{$\mathrm{cores}$}(m-2-2);
 \end{tikzpicture}
\]
commutes up to a natural isomorphism, or equivalently, such that
\[
 \begin{tikzpicture}
  \matrix(m)[diagram]{A\text{-mod} & B\text{-mod}^H \\
 A\text{-mod}_H & B\text{-mod} \\};
\draw[->]
 (m-1-1) edge (m-1-2)
 (m-2-1) edge node[auto]{$\mathrm{res}$}(m-1-1)
 (m-2-1) edge  (m-2-2)
(m-2-2) edge node[auto,swap]{$\mathrm{coind}$}(m-1-2);
 \end{tikzpicture}
\]
commutes up to a natural isomorphism.

The simplest example is when $A=k$ is trivial ($k$ is the base field) and $B=H$. The category $H\text{-mod}^H$ is called the category of Hopf modules of $H$. A linear equivalence $F:k\text{-mod}\to H\text{-mod}^H$ making the diagram
\[
 \begin{tikzpicture}
  \matrix(m)[diagram]{k\text{-mod} & H\text{-mod}^H \\
 H\text{-mod} & H\text{-mod} \\};
\draw[->]
 (m-1-1) edge node[auto]{$F$}(m-1-2)
         edge node[auto,swap]{$\mathrm{ind}$}(m-2-1)
 (m-2-1) edge node[auto]{$=$} (m-2-2)
(m-1-2) edge node[auto]{$\mathrm{cores}$}(m-2-2);
 \end{tikzpicture}
\]
commutative (up to a natural isomorphism) is due to Sweedler \cite{Sw}; $F$ is given simply by $F(V)=H\otimes V$.

Let us notice that any $H$-Morita equivalence is given by a vector space $M$ which is a right $A$-module and left $B$-module, satisfying the compatibility relation $b\cdot(m\cdot a)=(b_{(1)}\cdot m)\cdot (b_{(2)}\cdot a)$ for all $a\otimes m\otimes b\in A\otimes M\otimes B$, where $b\mapsto b_{(1)}\otimes b_{(2)}\in B\otimes H$ is the $H$-comodule structure of $B$. Under this condition the space $M\otimes H$ is a $B$-$A$-bimodule, given by $b\cdot (m\otimes h)=(b_{(1)}\cdot m)\otimes (b_{(2)}\cdot h)$, $(m\otimes h)\cdot a=(m\cdot (h_{(2)}\cdot a))\otimes h_{(1)}$. The functor $A\text{-mod}\to B\text{-mod}^H$ is given by $V\mapsto V\otimes_A(M\otimes H)$, and $A\text{-mod}_H\to B\text{-mod}$ by $W\mapsto \bigl(W\otimes_A(M\otimes H)\bigr)^H$, where $X^H$ (for a $H$-module $X$) is $\{x\in X; (\forall h\in H)\, h\cdot x=\epsilon(h)x\}$.

\subsection{Conjectural answer to Question 1}
Suppose again that $G$ is a connected Lie group and its Lie algebra $\mathfrak{g}$ has an invariant inner product $\langle,\rangle$. Let $\mathfrak{r},\mathfrak{b}\subset\mathfrak{g}$ be a Manin triple and let $R$ and $B$ be the corresponding Poisson-Lie groups. Suppose also that $\mathfrak{v}$ is another Lagrangian subalgebra with the property that $\mathfrak{v}\cap\mathfrak{r}$ is the Lie algebra of a closed connected subgroup $R_\mathfrak{v}\subset R$ and similarly, $\mathfrak{v}\cap\mathfrak{b}$ is the Lie algebra of a closed connected subgroup $B_\mathfrak{v}\subset B$. By Drinfeld's classification of Poisson homogeneous spaces \cite{D} the homogeneous space $R/R_\mathfrak{v}$ has a Poisson structure such that the map $R\times (R/R_\mathfrak{v})\to R/R_\mathfrak{v}$ is Poisson, and similarly for $B/B_\mathfrak{v}$.

Below we shall prove a symplectic version of the following loosely stated conjecture: If Hopf algebra $H$ is a (suitable) quantization of the Manin triple $\mathfrak{r},\mathfrak{b}\subset\mathfrak{g}$ and algebras $A$ and $B$ are (suitable) quantizations of the Poisson manifolds $R_\mathfrak{v}$ and $B_\mathfrak{v}$ respectively, then $A$ and $B$ are $H$-Morita equivalent.

In particular, if $\mathfrak{v}=\mathfrak{r}$ we get $A=k$ and $B=H$, i.e.\ Sweedler's example of $H$-Morita equivalence.

Let us notice that we prove the symplectic version of the conjecture only under the additional assumption that $\mathfrak{v}$ is transverse to both $\mathfrak{r}$ and $\mathfrak{b}$. The general case would require colored surfaces where we allow adjacent subalgebras to have non-trivial intersection. The corresponding colored $G$-bundles would have a reduction over the corresponding corner to the group exponentiating the intersection. We shall treat this generalization elsewhere.

\subsection{Symplectic $H$-Morita equivalence}

In this final section we provide a symplectic analogue of our conjectural answer to Question 1. Vector spaces are replaced by symplectic manifolds of the form $\mathcal{M}_\Sigma$ as follows:
$$\picSHMor$$
In other words, $A$ is replaced by the symplectic groupoid integrating the Poisson homogeneous space $R_\mathfrak{v}$, $B$ by the symplectic groupoid integrating $B_\mathfrak{v}$, $H$ by the double symplectic groupoid integrating both $R$ and $B$, and $M$ by the moduli space of the displayed triangle.

Let us recall the definition of Morita equivalence of symplectic groupoids \cite{X}.
Let $\mathsf{I}$ denote the groupoid with two objects $0$ and $1$ and with a unique morphism $0\to1$. Let $\Gamma$ be a Lie groupoid with a groupoid morphism $\Gamma\to\mathsf{I}$. $\Gamma$ splits naturally to 4 components: $\Gamma_{ij}$ ($i,j\in\{0,1\}$) is the space of arrows lying over the unique morphism $i\to j$. Let $M_i$ denote the space of objects of $\Gamma$ lying over $i\in\{0,1\}$. If the maps $\Gamma_{01}\to M_1$ and $\Gamma_{10}\to M_0$ are surjective  then $\Gamma_{00}$ and $\Gamma_{11}$ are said to be Morita equivalent via the bimodules $\Gamma_{01}$ and $\Gamma_{10}$. For Morita equivalence of symplectic groupoids the groupoid $\Gamma$ is required to be symplectic.

Let $\mathfrak{r},\mathfrak{b},\mathfrak{v}\subset\mathfrak{g}$ be as above and let $R,B,V\subset G$ be the corresponding groups.
Let $M_{0,0}=R\times B$ and let the arrows $(r_1,b_1)\to(r_2,b_2)$ in $\Gamma_{00}$ be $(v,r)\in V\times R$ such that $r_1b_1v=b\,r_2b_2$; composition of arrows is by $(v,r)(v',r')=(vv',rr')$. The groupoid $\Gamma_{00}$ can be seen as $\mathcal{M}_\Sigma$ for the surface
$$\picGammaZZ$$
which makes it to a symplectic groupoid. The groupoid composition is by (horizontal) gluing of rectangles.

The symplectic groupoid $\Gamma_{00}$ integrates the following Poisson structure on $R\times B$. We have the Poisson action $B\times B_\mathfrak{v}\to B_\mathfrak{v}$. The forgetful functor
$$\text{Poisson manifolds with a moment map to } R \to\text{Poisson $\mathfrak{b}$-manifolds}$$
has a right adjoint $F$ (see \cite{Se2}), and $F(B_\mathfrak{v})=R\times B$ is our Poisson manifold. It is a semi-classical analog of the crossed product $A\rtimes H$, where $A$ is a quantization of $B_\mathfrak{v}$ (an associative algebra) and $H$ a quantization of the Lie bialgebra $\mathfrak{b}$ (a Hopf algebra). Notice that $A\rtimes H\text{-mod}$ is equivalent to $A\text{-mod}_H$.

Let us suppose for simplicity that the map $R\times B\to G$, $(r,b)\mapsto rb$, is a diffeomorphism. The symplectic groupoid $\Gamma_{00}$ is Morita equivalent to the symplectic groupoid $\Gamma_{11}$ integrating the Poisson manifold $R_v$, i.e.\ $\mathcal{M}_{\Sigma}$ for the surface
$$\picGammaoo$$
The bimodule $\Gamma_{01}$ is $\mathcal{M}_{\Sigma}$ for the surface
$$\picGammaZo$$
with the bimodule structure given by gluing along the vertical sides.

This Morita equivalence of symplectic groupoids is analogous to a Morita equivalence $A\text{-mod}_H\to B\text{-mod}$. If we exchange $\mathfrak{r}$ and $\mathfrak{b}$ then we get a similar Morita equivalence, analogous to $A\text{-mod}\to B\text{-mod}^H$, and these two Morita equivalences of symplectic groupoids are easily seen to form a commutative square analogous to $H$-Morita equivalence of $A$ and $B$. Finally, the moduli space for the triangle
$$\picTriTwo$$
is analogous to the vector space $M$.

In the case when $R,B=\mathbb{T}^n$, $G=R\times B$, all these symplectic manifolds are of the form $\mathbb{R}^{2m}/\mathbb{Z}^k$, with constant symplectic structure. They can be therefore easily geometrically quantized (provided the symplectic form is integral on $\mathbb{Z}^k$) and these quantizations are compatible with the groupoid/module structures, so we get an $H$-Morita equivalence. If we choose the Planck constant so that the quantization $H$ of the double symplectic groupoid $R\times B$ is trivial, we get the standard proof of Morita equivalence of quantum tori \cite{RS}.

\end{document}